\documentclass[12pt, 14paper,reqno]{amsart}
\usepackage{amsmath,amsfonts,amssymb}
\usepackage[breaklinks]{hyperref}
\usepackage{graphicx}

\theoremstyle{plain}
\newtheorem{thm}{Theorem}[section]

\newtheorem{cor}{Corollary}[section]

\theoremstyle{proof}

\numberwithin{equation}{section}

\begin{document} 
\title[Quadratic fields with $n$-rank at least 2]{Parameterized families of quadratic fields with $n$-rank at least $2$}
\author{Azizul Hoque and Srinivas Kotyada }
\date{\today}
\maketitle

\begin{abstract}
We construct parameterized families of imaginary (resp. real) quadratic fields whose class groups have $n$-rank at least $2$. 
\end{abstract}

\section{Introduction}
It is well-known that there exist infinitely many real and imaginary quadratic fields each with class number divisible by a given positive integer. In particular, Nagell \cite{N22} (resp. Yamamoto \cite{Y70} and Weinberger \cite{W73}) proved that there are infinitely many imaginary (resp. real) quadratic fields each with class number divisible by an integer $n\geq 2$. Numerous results on parametric families of imaginary (resp. real) quadratic fields each with class number divisible by $n\geq 2$ are also proved by many authors (cf. \cite{CM03, CHKP18, CH20, CH23, H21, H22, Y02}).  
 A closely related problem concerns the $n$-rank of class groups of quadratic fields (in fact, any number fields). This problem has attracted quite a bit of attention. Here by a $n$-rank of the class group $G$ of a number field, we mean the largest positive integer $n$ for which $G$ contains a subgroup $H$ such that $H\cong\left(\mathbb{Z}/p\mathbb{Z}\right)^r$. Clearly if $G$ has an element of order $n$ then the $n$-rank of $G$ is at least $1$. Under the assumption that the class groups of quadratic fields have $n$-rank at least $2$, Cohen and Lenstra \cite{CL84} predict that the existence of a positive proportion of such fields. On the other hand, Yamamoto \cite{Y70} proved the existence of infinitely many imaginary quadratic fields  whose class groups have $n$-rank at least $2$ for any integer odd integer $n$.  Applying this result of Yamamoto,  Yu \cite{Y02} gave a parametric family of imaginary quadratic fields whose class group has $n$-rank at least  $2$ for some odd integer $n\geq 3$. 
Not much is known about the quadratic fields with higher rank class groups. In 1973, Craig \cite{C73} proved the existence of infinitely many imaginary quadratic fields with class groups having $3$-rank at least $4$. On the other hand, Llorente and Quer \cite{LQ88, Q87} gave three imaginary quadratic fields with class groups having $3$-rank $6$. It was Kishi and Komatsu \cite{KK17}, who gave a parametric family of imaginary quadratic fields with class groups having $3$-rank at least $3$. However, the fields in this family is very sparse since the discriminants of these fields grow exponentially. It seems that there is no parametric family of imaginary quadratic fields with class groups having $n$-rank at least $2$, for arbitrary positive integer $n$, is not known. On the other hand, very less is known about the $n$-rank of the class groups of real quadratic fields except for $n=3, 5,7$ which are known due to Crain \cite{C77} and Mestre \cite{M80, M81}. 

The goal of this article is to give a simply parameterized family of imaginary quadratic fields whose class groups have $n$-rank at least $2$ (see, Theorem \ref{thmi}) for arbitrary positive integer $n$. We also give an analogous family in case of real quadratic fields (see, Theorem \ref{thmr}).    


\section{Families of imaginary quadratic fields: $n$-rank}
Let $a, b, c$ and $n$ be any positive integers. Assume that 
\begin{equation}\label{d}
d:=d(a,b,c,n)=2(a^n+b^n)c^n-\left(\sum_{i=1}^{n}a^{n-i}b^{i-1}\right)^2-(a-b)^2c^{2n}.
\end{equation}
We can choose $a, b, c$ in such a way that $d$ is square-free and positive. Here, our goal is to show that 
the class group of $\mathbb{Q}(\sqrt{-d})$ contains a subgroup which is isomorphic to $\mathbb{Z}/n\mathbb{Z}\times \mathbb{Z}/n\mathbb{Z}$. More precisely, we prove the following:
\begin{thm}\label{thmi}
Let $a, b, c$ and $n$ be any positive integers such that $a\ne b$, and $d$ as defined by \eqref{d} is positive and square-free.
Assume that $d\geq 4(abc^2)^{\ell}$, where $\ell$ is the large prime divisor of $n$. Then the class group of $\mathbb{Q}(\sqrt{-d})$ has $n$-rank at least $2$. 
\end{thm}
We note that if $n$ itself is a prime then the assumption $d>4(abc^2)^\ell$ can be reduced to $d>4abc^2$ which allows more members in the family. Again if the positive integers $a$ and $b$ are equal, but other than $1$, then the $d$ is not square-free and thus the proof of Theorem \ref{thmi} does not work. In fact, in this case it is not easy to show the $n$-rank of the corresponding class group is at least $2$ even if it exists. On the other hand, it is easy to see that the $n$-rank of the class group of $\mathbb{Q}(\sqrt{-d})$ is at least one when $a=b=1$. In this case, the proof follows from that of Theorem \ref{thmi}. More precisely, we have the following:
\begin{cor}\label{cor1}
For any two positive integers $c$ and $n$, the $n$-rank of the class group of the imaginary quadratic field $\mathbb{Q}(\sqrt{n^2-4c^n})$ is at least one provided $4c^n-n^2>c^\ell$ and is square-free.
\end{cor}
One can put $x=n$ and $y=c$, and then apply the arguments that used in the proof of Theorem \ref{thmi} to get this corollary. Theorem \ref{thmi} extends \cite[Lemma 2.1]{Y20}, which is one of the main results of that paper.  

\begin{proof}[Proof of Theorem \ref{thmi}]
We first assume that 
\begin{equation*}
x_1=\sum_{i=1}^{n}a^{n-i}b^{i-1}+(a-b)c^n\text{ and } x_1=ac.
\end{equation*}
Then 
\begin{equation}\label{x1}
x_1^2-4y_1^n=-d.
\end{equation}
Similarly we put  $$x_2=\sum_{i=1}^{n}a^{n-i}b^{i-1}+(b-a)c^n\text{ and } y_2=bc,$$
to get
\begin{equation}\label{x2}
x_2^2-4y_2^n=-d.
\end{equation}
Now \eqref{x1} can be written as follows:
\begin{equation}\label{y1n}
\left(\frac{x_1+\sqrt{-d}}{2}\right) \left(\frac{x_1-\sqrt{-d}}{2}\right)=y_1^n.
\end{equation}
This shows that $x_1, y_1\in \left(\frac{x_1+\sqrt{-d}}{2}, \frac{x_1-\sqrt{-d}}{2}\right)$. Let $p$ be a prime such that  $p\mid \gcd(x_1, y_1)$, then by \eqref{x1}, $p^2\mid d$ which contradicts the fact that $d$ is square-free. Thus $\gcd(x_1, y_1)=1$ and hence $1\in \left(\frac{x_1+\sqrt{-d}}{2}, \frac{x_1-\sqrt{-D}}{2}\right)$ which further implies that $\left(\frac{x_1+\sqrt{-d}}{2}, \frac{x_1-\sqrt{-d}}{2}\right)=\mathcal{O}_d$, the ring of integers of $\mathbb{Q}(\sqrt{-d})$. In other words, the ideals $\left(\frac{x_1+\sqrt{-d}}{2}\right) $ and $\left( \frac{x_1-\sqrt{-d}}{2}\right)$ are coprime, and thus using \eqref{y1n} we can write 
\begin{equation}\label{1n}
\left(\frac{x_1+\sqrt{-d}}{2}\right)=\mathfrak{a}^n,~~\left(\frac{x_1-\sqrt{-d}}{2}\right)=\bar{\mathfrak{a}}^n
\end{equation}
for some ideal $\mathfrak{a}$ and its conjugate $\bar{\mathfrak{a}}$.

In the similar fashion, using \eqref{x2} one can find an ideal $\mathfrak{b}$ and its conjugate $\bar{\mathfrak{b}}$  such that 
\begin{equation}\label{2n}
\left(\frac{x_2+\sqrt{-d}}{2}\right)=\mathfrak{b}^n, ~~ \left(\frac{x_2-\sqrt{-d}}{2}\right)=\bar{\mathfrak{b}}^n.
\end{equation}
Let $\text{Cl}(\mathfrak{a})$ and Cl$(\mathfrak{b})$ be the ideal classes in $\mathbb{Q}(\sqrt{-d})$ containing $\mathfrak{a}$ and $\mathfrak{b}$ respectively. We will show that $\text{Cl}(\mathfrak{a})$ and Cl$(\mathfrak{b})$ generate a subgroup of the class group of $\mathbb{Q}(\sqrt{-d})$ which is isomorphic to $\mathbb{Z}/n\mathbb{Z}\times \mathbb{Z}/n\mathbb{Z}$. For this, it is sufficient to show that  $\text{Ord}\left(\text{Cl}(\mathfrak{a})\right)=\text{Ord}\left(\text{Cl}(\mathfrak{b})\right)=n$, and $\text{Cl}(\mathfrak{a})$ and $\text{Cl}(\mathfrak{b})$ are distinct. In other words, it is sufficient to show that none of $ \mathfrak{a}^p$, $\mathfrak{b}^q $, $\mathfrak{a}^p{\mathfrak{b}}^q$  and $\mathfrak{a}^p\bar{\mathfrak{b}}^q$ is principal for any prime divisors $p$ and $q$ of $n$.

We first show that $\mathfrak{a}^p$ is not principal for any prime divisor $p$ of $n$. On the contrary, we
assume that $$\mathfrak{a}^p=\left(\frac{u+v\sqrt{-d}}{2}\right)$$ for some $u, v\in \mathbb{Z}$ with $u\equiv v\pmod 2$. This can be rewritten as 
$$\mathfrak{a}^n=\left(\frac{u+v\sqrt{-d}}{2}\right)^{n/p}.$$
Applying \eqref{y1n} and then equating the imaginary parts, we see that $v\mid 1$, and thus $v \ne 0$. Taking norm on both sides of the last equation, we get 
$$y_1^n=\left(\frac{u^2+dv^2}{4}\right)^{n/p}\geq \left(\frac{u^2+d}{4}\right)^{n/p}.
$$ Since
$d \geq (4abc^2)^{\ell}$, so that $y_1^n> (abc^2)^{\ell n/p}>(ac)^{\ell n/p}=y_1^{\ell n/p}$, which is not possible. Thus $\mathfrak{a}^p$ is not principal. Similarly, one can show that 
$\mathfrak{b}^q$ is not principal for any prime divisor $q$ of $n$.  

We now claim that $ \mathfrak{a}^p\mathfrak{b}^q $ is not principal for any prime divisors $p$ and $q$ of $n$. If not then, we can find some $u, v\in \mathbb{Z}$ with same parity such that 
\begin{equation}\label{pq}
\mathfrak{a}^p\mathfrak{b}^q=\left(\frac{u+v\sqrt{-d}}{2}\right). 
\end{equation}
Using \eqref{1n} and \eqref{2n}, we have
$$\mathfrak{a}^{pn}\mathfrak{b}^{qn}=\frac{1}{2^{pq}}\left( \sum_{r=0}^{\frac{p-1}{2}} \sum_{s=0}^{\frac{q-1}{2}}A_{r,s}+B_{r,s}\sqrt{-d}
\right),$$
where 
\begin{align*}
A_{r,s}&=\left(\binom{p}{2r} \binom{q}{2s} x_1^{p-2r}x_2^{q-2s}-\binom{p}{2r+1} \binom{q}{2s+1} x_1^{p-2r-1}x_2^{q-2s-1}d\right)(-d)^{r+s}\\
B_{r,s}&=\left(\binom{p}{2r} \binom{q}{2s+1} x_1^{p-2r}x_2^{q-2s-1}+\binom{p}{2r+1} \binom{q}{2s} x_1^{p-2r-1}x_2^{q-2s}\right)(-d)^{r+s}.
\end{align*} 
Therefore as $\pm 1$ are the only units in $\mathbb{Q}(\sqrt{-d})$ (using \eqref{pq}), $v\ne 0$. Again writing \eqref{pq} in norm form, we get 
$$y_1^{pn}y_2^{qn}=\left(\frac{u^2+v^2d}{4}\right)^n\geq \left(\frac{u^2+d}{4}\right)^n.$$
Since $d\geq 4(abc^2)^{\ell}$, so that $y_1^py_2^q> (abc^2)^{\ell }=y_1^{\ell }y_2^{\ell }.$
This is not possible and hence  $\mathfrak{a}^p\mathfrak{b}^q $ is not principal. Similarly, one can show that $\mathfrak{a}^p\bar{\mathfrak{b}}^q$ is not principal. This completes the proof. 
\end{proof}
\section{Families of real quadratic fields: $n$-rank}
Let $a, b$ and $c$ be any positive integers. For positive integer  $n$, we define 
\begin{equation}\label{rd}
\Delta:=\Delta(a,b,c,n)=\left(\sum_{i=1}^{n}a^{n-i}b^{i-1}\right)^2+(a-b)^2c^{2n}-2(a^n+b^n)c^n.
\end{equation}
We will show that if $d$ is square-free, positive and sufficiently large, then the class group of $\mathbb{Q}(\sqrt{\Delta})$ contains  $\mathbb{Z}/n\mathbb{Z}\times \mathbb{Z}/n\mathbb{Z}$ as a subgroup.  Precisely we prove:
\begin{thm}\label{thmr}
Let $a, b$  and $c$ be any distinct positive integers. For any positive integer $n\equiv 0\pmod 4$, let $d$ be as defined by \eqref{rd}.  Assume that $d$ is square-free and $\geq 16(abc^2)^{2\ell}$, where $\ell$ is the large prime divisor of $n$. Then the class group of $\mathbb{Q}(\sqrt{\Delta})$ has $n$-rank at least $2$. 
\end{thm}

\begin{proof}
We need only outline the proof of Theorem \ref{thmr}, since in most aspects it is similar to the proof of Theorem \ref{thmi}. We first need to construct two integral ideals $\mathfrak{A}$ and $\mathfrak{B}$ such that 
\begin{equation}\label{rab}
\left(\frac{x_1+\sqrt{d}}{2}\right)=\mathfrak{A}^n,~~ \left(\frac{x_2+\sqrt{d}}{2}\right)=\mathfrak{B}^n.
\end{equation}
We then show that for any prime divisors $p$ and $q$ of $n$, $ \mathfrak{A}^p$, $\mathfrak{B}^q $, $\mathfrak{A}^p{\mathfrak{B}}^q$  and $\mathfrak{A}^p\bar{\mathfrak{B}}^q$ are not principal. Here $\bar{\mathfrak{B}}$ is the conjugate of $\mathfrak{B}$. The first two cases follow from the fact that there are no rational integers $u$ and  $v$ with same parity such that 
$y_1^{pn}=\left(\frac{u^2-dv^2}{4}\right)^{n}$. In fact, this follows from $\left(\frac{u^4-dv^2(dv^2-2u^2)}{16}\right)^{n/2}> \left(\frac{d}{16}\right)^{n/2} $ since $n/2$ is even, and then using the hypothesis $d\geq 16(abc^2)^{2\ell}$. 

Similarly, the last two cases follow from the fact that there are no rational integers $u$ and  $v$ with same parity such that 
$y_1^{pn}y_2^{qn}=\left(\frac{u^2-dv^2}{4}\right)^{n}$. This argument follows from the contradiction given by $\left(\frac{u^4-dv^2(dv^2-2u^2)}{16}\right)^{n/2}> \left(\frac{d}{16}\right)^{n/2} $ since $n/2$ is even, and then by the use of the hypothesis $d\geq 16(abc^2)^{2\ell}=16y_1^{2\ell} y_2^{2\ell}$. 
\end{proof}

\section*{acknowledgements}
The authors are grateful to Professor Claude Levesque for introducing them to \cite{LQ88, Q87} and for supplying the copies of the same. A. Hoque  is supported by SERB MARTICS grant (No. MTR/2021/000762) and CRG grant (No. CRG/2023/007323), Govt. of India and K. Srinivas acknowledges the grant `Advanced Research and Education in Mathematical Sciences at IMSc-II funded by DAE', Govt. of India. K. Srinivas thanks the organizers of the conference ICCGNFRT-2021 for the invitation where this work was presented and it is available at \url{<https://www.youtube.com/watch?v=z4Toa80AMik>}. The authors are grateful to anonymous referee for careful reading and helpful comments.

\end{document}